\documentclass[12pt,twoside,letterpaper]{article} 
\usepackage{times,fancyhdr}   
\usepackage{amsfonts}                        
\usepackage{amsthm}                         
\usepackage{amsmath}                   
\usepackage{amssymb}  
\usepackage[normalem]{ulem}
\usepackage{cases}
\usepackage{bm}

\usepackage{mathtools}

\numberwithin{equation}{section}


\makeatletter 
\def\cleardoublepage{\clearpage\if@twoside \ifodd\c@page\else%
    \hbox{}%
    \thispagestyle{empty}%
    \newpage%
    \if@twocolumn\hbox{}\newpage\fi\fi\fi} 
\makeatother 

\def \C{\mathbb{C}}

\def \Q{\mathbb{Q}}

\def \Z{\mathbb{Z}}

\def \M{\mathcal{M}}

\def \a{\alpha}
\def \b{\beta}

\DeclareMathOperator*{\lcm}{lcm}

\def \={\coloneqq}
\def \th{\theta}
\def \la{\lambda}

\def \cD{\mathcal{D}}

\def \cM{\mathcal{M}}
\def \cR{\mathcal{R}}
\def \cS{\mathcal{S}}
\def \cT{\mathcal{T}}
\def \cU{\mathcal{U}}
\def \cV{\mathcal{V}}

\theoremstyle{plain}

\newtheorem{thm}{Theorem}[section] 

\newtheorem{cor}{Corollary}[section]
\newtheorem{lemma}{Lemma}[section]

\newtheorem{example}{Example}[section]

\theoremstyle{definition}
\newtheorem{remark}{Remark}[section]

\begin{document}
\title{
{\begin{flushleft}
\vskip 0.45in
{\normalsize\bfseries\textit{}}
\end{flushleft}
\vskip 0.45in
\bfseries\scshape 
A linear independence criterion \\ for certain infinite series \\ with polynomial orders
}}

\thispagestyle{fancy}
\fancyhead{}
\fancyhead[L]{In: Book Title \\ 
Editor: Editor Name, pp. {\thepage-\pageref{lastpage-01}}} 
\fancyhead[R]{ISBN 0000000000  \\
\copyright~2007 Nova Science Publishers, Inc.}
\fancyfoot{}
\renewcommand{\headrulewidth}{0pt}

\author{\bfseries\itshape 
Shinya Kudo
\thanks{Hirosaki University, Graduate School of Science and Technology, 
Hirosaki 036-8561, Japan \newline e-mail:gongtengs48@gmail.com }} 

\date{}
\maketitle

\begin{abstract}
Let $q$ be a Pisot or Salem number. 
Let $f_j(x)$ $(j=1,2,\dots)$ be integer-valued polynomials 
of degree $\ge2$ with positive leading coefficients, and let
$\{a_j (n)\}_{n\ge1}$ $(j=1,2,\dots)$ be sequences of algebraic integers 
in the field $\Q(q)$ with suitable growth conditions.
In this paper, we investigate linear independence over $\Q(q)$ of the numbers
\begin{equation*}
1,\qquad \sum_{n=1}^{\infty} \frac{a_j (n)}{q^{f_j (n)}} \quad (j=1,2,\dots).
\end{equation*}
In particular, when $a_j(n)$ $(j=1,2,\dots)$ are polynomials of $n$,
we give a linear independence criterion for the above numbers.
\end{abstract}
 


\pagestyle{fancy}  
\fancyhead{}
\fancyhead[EC]{Shinya Kudo}
\fancyhead[EL,OR]{\thepage}
\fancyhead[OC]{
A linear independence criterion for certain infinite series with polynomial orders
}
\fancyfoot{}
\renewcommand\headrulewidth{0.5pt}

\section{Introduction and main results}\label{sec1}   

In 1957, Erd\H{o}s \cite{Erdos} proved that if $\ell \ge 1$ is an integer and $1<n_1<n_2<\cdots$ are integers with $\limsup_{k \to \infty} n_k/k^\ell = \infty$, 
then for any integer $b\ge2$, the number $\sum_{k=1}^{\infty} 1/b^{n_k}$ 
is transcendental or an algebraic number with degree at least $\ell+1$.
For example, we see that if $f(x)$ is an integer-valued polynomial of degree $d \ge2$ 
with $1<f(1)<f(2)<\cdots$, then for any integer $b\ge2$, the number 
$\sum_{n=1}^{\infty} 1/b^{f(n)}$ is transcendental or an algebraic number with degree at least $d$.

On the other hand, Kumar \cite{Kumar} obtained linear independence result
for certain infinite series with monomial orders. More precisely,
he proved that if $k\ge2$ and $1\le a_1<a_2<\dots<a_m$ are integers with 
$\sqrt[k]{a_i/a_j} \notin \Q$ $(i \neq j)$, then for any integer $b\ge2$, the numbers
\begin{equation*}
1, \quad \sum_{n=1}^{\infty} \frac{1}{b^{a_1 n^k}}, \quad \sum_{n=1}^{\infty} \frac{1}{b^{a_2 n^k}},\dots,
\sum_{n=1}^{\infty} \frac{1}{b^{a_m n^k}} 
\end{equation*}
are linearly independent over $\Q$. Recently, Murakami and Tachiya \cite{Murakami, Murakami2} 
gave a generalization of Kumar's result.
In particular, they showed in \cite[Corollary 1.3]{Murakami} that
for any integer $b \ge 2$, the set of the numbers
\begin{equation}
1, \qquad \sum_{n=1}^{\infty} \frac{1}{b^{i n^j}} \quad (i=1,2,\dots, \ j=2,3,\dots) \label{0.mono}
\end{equation}
is linearly independent over $\Q$.

In this paper, we study linear independence of certain infinite series with polynomial orders, 
and as a corollary, get a generalization of the linear independence result
for the numbers \eqref{0.mono}.
Moreover, we give a linear independence criterion under specific conditions.

Before stating our results, we need some preparation.
We write $f(x) \equiv g(x)$ if the polynomials $f(x)$ and $g(x)$ are identically equal,
and $f(x) \not\equiv g(x)$ otherwise.
Recall that the polynomial $f(x)$ is called an integer-valued polynomial 
if the value $f(n)$ is an integer for every positive integer $n$.
Note that all coefficients of $f(x)$ are rational numbers if $f(x)$ is an integer-valued polynomial.
In addition, a Pisot (resp. Salem) number are defined as a real algebraic integer greater than one,
all of whose Galois conjugates other than itself have absolute value less than one 
(resp. less than or equal to one, and at least one conjugate has absolute value exactly one).
Clearly, all of integers greater than one are Pisot numbers.
Our results are the following.


\begin{thm}\label{thm1}
Let $q$ be a Pisot or Salem number. 
Let $\ell \ge 1$ be an integer and $f_j(x)$ $(j=1,2,\dots,\ell)$ be integer-valued polynomials 
of degree $\ge2$ with positive leading coefficients.
Let $\{a_j (n)\}_{n\ge1}$ $(j=1,2,\dots,\ell)$ be sequences of algebraic integers in the field $\Q(q)$
with $\log(1+|a_j (n)^\sigma|)=o(n)$ for every embedding $\sigma \colon \Q(q) \to \C$.
Suppose that for any non-empty subset $\cS$ of $\{1,2,\dots,\ell \}$,
there exist integers $i \in \cS$ and $A$ satisfying the following two conditions:
\begin{enumerate}
\setlength{\parskip}{0pt}
\setlength{\itemsep}{0pt}
\item \label{i}
$f_i (x+A) \not\equiv f_j (Bx+C)+D$ 
for any integers $j \in \cS \setminus \{i\}$, $C$, $D$ and any positive rational number $B$.
\item \label{ii}
There exists a positive integer $E$ such that $\displaystyle \liminf_{n \to \infty} |a_i (En+A)| \neq 0$.
\end{enumerate}
Then, the numbers
\begin{equation}
1,\qquad \sum_{n=1}^{\infty} \frac{a_j (n)}{q^{f_j (n)}} \quad (j=1,2,\dots,\ell) \label{1.1.ind}
\end{equation}
are linearly independent over $\Q(q)$.
\end{thm}


We give the following corollary of Theorem \ref{thm1}.

\begin{cor}\label{cor1}
Let $q$ be a Pisot or Salem number.
Let $f_j(x)$ $(j = 1,2,\dots,\ell)$ be integer-valued polynomials 
of degree $\ge2$ with positive leading coefficients
such that the condition \eqref{i} in Theorem \ref{thm1} holds.
Then, for any sequences $\{a_j (n)\}_{n\ge1}$ $(j=1,2,\dots,\ell)$ of non-zero integers with 
\mbox{$\log(1+|a_j (n)|)=o(n)$}, the numbers
\begin{equation*}
1,\qquad \sum_{n=1}^{\infty} \frac{a_j (n)}{q^{f_j (n)}} \quad (j=1,2,\dots,\ell)
\end{equation*}
are linearly independent over $\Q(q)$.
In particular, so are the numbers
\begin{equation*}
1,\qquad \sum_{n=1}^{\infty} \frac{1}{q^{f_j (n)}} \quad (j=1,2,\dots,\ell).
\end{equation*}
\end{cor}


Let us observe the condition \eqref{i} in Theorem \ref{thm1}.
Define $f(x) \= x^2$, $g(x) \= (x+1)^2$ and $h(x) \= (2x)^2$. 
Clearly, we have
\begin{align*}
f(x+A) &\equiv g(x+A-1), \\
g(x+A) &\equiv f(x+A+1)
\end{align*}
for any integer $A$, and so the polynomials $f$ and $g$ do not satisfy the condition \eqref{i}.
On the other hand, for the polynomials $f$ and $h$, we have $f(x) \equiv h(x/2)$, while
\begin{equation*}
f(x+1) \not\equiv h(Bx+C)+D
\end{equation*}
for any positive rational number $B$ and any integers $C$, $D$.
Hence, the polynomials $f$ and $h$ satisfy the condition \eqref{i}.


Now, we apply Theorem \ref{thm1} to derive a linear independence result for the values of
the Jacobi theta constants defined in $|z|<1$ by
\begin{equation}
\th_2 (z) \= 2z^{\frac{1}{4}} \sum_{n=0}^\infty z^{n(n+1)}, \quad 
\th_3 (z) \= 1+2 \sum_{n=1}^\infty z^{n^2}, \quad
\th_4 (z) \= 1+2 \sum_{n=1}^\infty (-1)^n z^{n^2}. \label{1.1.Jacobi}
\end{equation} 
Let $f_1(x) \= x(x+1)$, $f_2(x) \= (2x)^2$ and $f_3(x) \= (2x-1)^2$.
Then, we immediately see that the polynomials $f_1(x)$, $f_2(x)$ and $f_3(x)$ 
satisfy the condition \eqref{i} in Theorem \ref{thm1}.
Let $q$ be a Pisot or Salem number, and $\ell \ge 0$ be an integer.
We define 
\begin{equation}
\a_{i,j} \= \sum_{n=1}^{\infty} \frac{n^i}{q^{f_j(n)}} \quad (i=0,1,\dots,\ell, \ j=1,2,3). \label{1.1.a_i,j}
\end{equation}
For any algebraic integers $A_{i,j}$ $ (i=0,1,\dots,\ell, \ j=1,2,3)$ in $\Q(q)$ not all zero,
we consider the linear combination over $\Q(q)$ of the numbers \eqref{1.1.a_i,j} given by
\begin{equation}
\sum_{j=1}^{3} \sum_{i=0}^{\ell} A_{i,j} \a_{i,j} 
= \sum_{j \in \cT} \sum_{n=1}^{\infty} \frac{a_j(n)}{q^{f_j(n)}},
\label{1.1.Aa_i,j}
\end{equation}
where 
\begin{equation*}
\cT \= \left\{ j \in \{1,2,3\} \mid \text{there exists an integer } i \ (0 \le i \le \ell) 
\text{ such that } A_{i,j} \neq 0 \right\} \ (\neq \emptyset)
\end{equation*}
and $a_j(n) \= \sum_{i=0}^{\ell} A_{i,j} n^i$ $(n \ge 1)$.
Clearly, the sequences $\{a_j (n)\}_{n\ge1}$ $(j \in \cT)$ satisfy 
the conditions in Theorem \ref{thm1}.
Thus, by Theorem \ref{thm1}, the numbers \eqref{1.1.Aa_i,j} does not belong to $\Q(q)$, namely,
the numbers $1, \ \a_{i,j}$ $(i=0,1,\dots,\ell, \ j=1,2,3)$ are linearly independent over $\Q(q)$.
Also, for any algebraic integers $u$ and $v$ in $\Q(q)$, we have
\begin{equation*}
u \sum_{n=1}^{\infty} \frac{n^k}{q^{n^2}} + v \sum_{n=1}^{\infty} \frac{(-1)^n n^k}{q^{n^2}}
= (u+v) \sum_{n=1}^{\infty} \frac{(2n)^k}{q^{(2n)^2}} 
+ (u-v) \sum_{n=1}^{\infty} \frac{(2n-1)^k}{q^{(2n-1)^2}}
\quad (k=0,1,\dots),
\end{equation*}
and hence we get the following example of Theorem \ref{thm1}.

\begin{example}\label{exm1}
Let $q$ be a Pisot or Salem number. Then, for any integer $\ell \ge 0$, the numbers
\begin{equation*}
1,\qquad \sum_{n=1}^{\infty} \frac{n^i}{q^{n(n+1)}}, 
\quad \sum_{n=1}^{\infty} \frac{n^j}{q^{n^2}}, 
\quad \sum_{n=1}^{\infty} \frac{(-1)^n n^k}{q^{n^2}} \quad (i,j,k=0,1,\dots,\ell)
\end{equation*}
are linearly independent over $\Q(q)$.
Thus, so are the values
\begin{align*}
1,\qquad \left. \frac{d^i}{dz^i} \left( z^{-\frac{1}{4}} \th_2 (z) \right) \right|_{z=1/q}, 
\quad \left. \frac{d^j}{dz^j} \th_3 (z) \right|_{z=1/q}, 
\quad \left. \frac{d^k}{dz^k} \th_4 (z) \right|_{z=1/q} \quad (i,j,k=0,1,\dots,\ell).
\end{align*}
\end{example}


\begin{remark}\label{rem1}
Bertrand \hfill \cite{Bertrand}, \hfill and \hfill independently \hfill 
Duverney, \hfill Ke. \hfill Nishioka, \hfill Ku. \hfill Nishioka \hfill and 
\\ Shiokawa \cite{Shiokawa} proved that 
for any algebraic number $\a$ $(0 \! < \! |\a| \! < \! 1)$ and any integer $m$ $(2 \le m \le 4)$, \\
the three values $\th_m(\a)$, $\cD \th_m(\a)$ and $\cD^2 \th_m(\a)$ 
are algebraically independent over $\Q$, where $\cD \= z (d/dz)$ is a differential operator,
by applying results of Nesterenko \cite{Nesterenko}.
On the other hand, for any above numbers $\a$ and $m$,
the four values $\th_m(\a)$, $\cD \th_m(\a)$, $\cD^2 \th_m(\a)$ and $\cD^3 \th_m(\a)$
are algebraically dependent over $\Q$
since $\th_m$ has an algebraic differential equation of the third order over $\Q$
(cf. \cite{Jacobi}).
In addition, Bertrand showed that for any algebraic number $\a$ $(0<|\a|<1)$,
any two values among the three values
$\th_2 (\a)$, $\th_3 (\a)$ and $\th_4 (\a)$ are algebraically independent over $\Q$.
Conversely, these three values are algebraically dependent over $\Q$
by the Jacobi identity $\th_3^4(z) = \th_2^4(z) + \th_4^4(z)$ (cf. \cite{Lawden}).
By the way, Elsner and Kumar \cite{Elsner} recently derived linear independence results
for functions and values related to $\th_3$.
\end{remark}


For any integer $\ell \ge 2$, we find that the polynomials 
\begin{equation*}
f_{i,j}(x) \= ix^j  \quad (i=1,2,\dots,\ell, \ j=2,3,\dots,\ell)
\end{equation*}
satisfy the condition \eqref{i} in Theorem \ref{thm1} 
(see Lemma \ref{lem2} in Section \ref{sec5}).
By Corollary \ref{cor1}, we get the following example.

\newpage

\begin{example}\label{exm2}
Let $q$ be a Pisot or Salem number.
Then, for any integer $\ell \ge 2$ and
any sequences $\{a_{i,j} (n)\}_{n\ge1}$ $(i=1,2,\dots,\ell, \ j=2,3,\dots,\ell)$ of non-zero integers with 
$\log(1+|a_{i,j} (n)|)=o(n)$, the numbers
\begin{equation*}
1,\qquad \sum_{n=1}^{\infty} \frac{a_{i,j} (n)}{q^{in^j}} \quad (i=1,2,\dots,\ell, \ j=2,3,\dots,\ell)
\end{equation*}
are linearly independent over $\Q(q)$.
\end{example}

\noindent
This generalizes a result of Karmakar, Kumar and Thangadurai \cite[Theorem 1.2]{Karmakar}. \\


Next, we give a linear independence criterion (Theorem \ref{thm2}) for the numbers \eqref{1.1.ind}
when $a_j(n)$ $(j=1,2,\dots,\ell)$ are polynomials of $n$.
Before stating Theorem \ref{thm2}, we need some preparation.

Let $q$ be a Pisot or Salem number.
The sequences $\{ b_j(n) \}_{n \ge 0}$ $(j=1,2,\dots,\ell)$ of complex numbers
are called linearly dependent over $\Q(q)$ 
if there exist algebraic numbers $k_j$ $(j=1,2,\dots,\ell)$ in $\Q(q)$ not all zero such that
$\sum_{j=1}^{\ell} k_j \cdot b_j(n) = 0$ for every integer $n \ge 0$.
Otherwise, the sequences $\{ b_j(n) \}_{n \ge 0}$ $(j=1,2,\dots,\ell)$ are called 
linearly independent over $\Q(q)$.
In addition, we define an equivalence relation $\sim$
on the set of all polynomials with rational coefficients.
For two polynomials $F(x)$ and $G(x)$ with rational coefficients,
we write $F \sim G$ if there exist rational numbers $b>0$, $c$ and an integer $d$ such that
\begin{equation*}
F(x) \equiv G(bx+c)+d,
\end{equation*}
and $F \not\sim G$ otherwise.
Clearly, this is an equivalence relation.

Let $f_j(x)$ $(j=1,2,\dots,\ell)$ be polynomials with rational coefficients such that
\mbox{$f_{j_1} \sim f_{j_2}$} for any integers $j_1$, $j_2$ $(1 \le j_1 \le j_2 \le \ell)$.
Then, there exists a polynomial $g(x)$ such that
there exist integers $B_j>0$, $C_j$ and $D_j$ $(j=1,2,\dots,\ell)$ satisfying
\begin{equation}
g(x) \equiv f_j \left( \frac{x+C_j}{B_j} \right) +D_j \quad (j=1,2,\dots,\ell). \label{1.1.gg}
\end{equation}
Indeed, by $f_1\sim f_j$ for every $j=1,2,\dots,\ell$,
there exist integers $r_j\ge1$, $s_j\ge1$, $t_j$ and $u_j$ $(j=1,2,\dots,\ell)$ such that
\begin{equation*}
f_1(x) \equiv f_j \left( \frac{s_j x+t_j}{r_j} \right) +u_j \quad (j=1,2,\dots,\ell).
\end{equation*}
Let $K$ be the least common multiple of $s_j$ $(j=1,2,\dots,\ell)$ and $g(x) \= f_1(x/K)$. Then,
\begin{equation*}
g(x) \equiv f_j \left( \frac{x+t_j v_j}{r_j v_j} \right) +u_j \quad (j=1,2,\dots,\ell),
\end{equation*}
where $v_j \= K/s_j$ $(j=1,2,\dots,\ell)$ are positive integers. 
Thus, $g(x)$ satisfies \eqref{1.1.gg}.
Theorem \ref{thm2} is the following.

\newpage


\begin{thm}\label{thm2} 
Let $\ell_1, \ell_2, \dots, \ell_m$ be positive integers. Let 
\begin{equation*}
\cU \= \bigcup_{i=1}^{m} \{ (i,j) \in \Z_{\ge1}^2 \mid j=1,2,\dots,\ell_i \}
\end{equation*}
and $f_{i,j}(x)$ $\bigl( (i,j) \in \cU \bigr)$ be integer-valued polynomials of degree $\ge2$
with positive leading coefficients such that 
\begin{equation*}
\begin{cases}
f_{i_1,j_1} \sim f_{i_2,j_2} & \text{ if }\ i_1=i_2, \\
f_{i_1,j_1} \not\sim f_{i_2,j_2} & \text{ if }\ i_1\neq i_2.
\end{cases}
\end{equation*}
Let $g_i(x)$ $(i=1,2,\dots,m)$ be polynomials given by \eqref{1.1.gg}, namely,
for each $i=1,2,\dots,m$, there exist integers $B_{i,j}>0$, $C_{i,j}$ and $D_{i,j}$ $(j=1,2,\dots,\ell_i)$ satisfying
\begin{equation}
g_i(x) \equiv f_{i,j} \left( \frac{x+C_{i,j}}{B_{i,j}} \right) +D_{i,j} \quad (j=1,2,\dots,\ell_i). \label{1.2.g_i}
\end{equation}
Let $q$ be a Pisot or Salem number, and 
$P_{i,j}(x)$ $\bigl( (i,j) \in \cU \bigr)$ be non-zero polynomials 
with algebraic integer coefficients in $\Q(q)$.
Then, the numbers
\begin{equation}
1,\qquad \sum_{n=1}^{\infty} \frac{P_{i,j} (n)}{q^{f_{i,j} (n)}} \quad \bigl( (i,j) \in \cU \bigr) \label{1.2.ind}
\end{equation}
are \hfill linearly \hfill independent \hfill over \hfill $\Q(q)$ \hfill
if \hfill and \hfill only \hfill if \hfill for \hfill each \hfill $i=1,2,\dots,m$, \hfill the \hfill sequences \\
$\{ \tilde{p}_{i,j}(n) \}_{n \ge 0}$ $(j=1,2,\dots,\ell_i)$ are linearly independent over $\Q(q)$, where
\begin{equation*}
\tilde{p}_{i,j}(n) \=
\begin{dcases*}
P_{i,j} \left( \frac{n+C_{i,j}}{B_{i,j}} \right) & if $\ n \equiv -C_{i,j} \pmod{B_{i,j}}$, \\
0 & otherwise
\end{dcases*} 
\qquad (j=1,2,\dots,\ell_i).
\end{equation*}
\end{thm}


In Theorem \ref{thm2}, if $\deg P_{i,j}(x)$ $\bigl( (i,j) \in \cU \bigr)$ are distinct, 
we immediately see that for each $i=1,2,\dots,m$, the sequences 
$\{ \tilde{p}_{i,j}(n) \}_{n \ge 0}$ $(j=1,2,\dots,\ell_i)$ are linearly independent over $\Q(q)$.
Hence, we have

\begin{cor}\label{cor2}
Let $q$ be a Pisot or Salem number. 
Let $f_j(x)$ $(j=1,2,\dots,\ell)$ be integer-valued polynomials of degree $\ge2$
with positive leading coefficients.
Let $P_j(x)$ $(j=1,2,\dots,\ell)$ be non-zero polynomials 
with algebraic integer coefficients in $\Q(q)$ and
$\deg P_{j_1}(x) \neq \deg P_{j_2}(x)$ for every integers $j_1,j_2$ $(1\le j_1 < j_2 \le \ell)$. 
Then, the numbers
\begin{equation*}
1,\qquad \sum_{n=1}^{\infty} \frac{P_j (n)}{q^{f_j (n)}} \quad (j=1,2,\dots,\ell)
\end{equation*}
are linearly independent over $\Q(q)$.
\end{cor}

\noindent
As an application of Theorem \ref{thm2}, we will give an example of linear dependence
(see Example \ref{exm5} in Section \ref{sec5}). \\


The present paper is organized as follows.
We will prove Theorem \ref{thm1} in Section \ref{sec3},
and Theorem \ref{thm2} in Section \ref{sec4}.
Before proving Theorem \ref{thm1}, we give Lemma \ref{lem1} in Section \ref{sec2}.
Lemma \ref{lem1} is the most important part to show Theorem \ref{thm1}.
Also, we will prove Theorem \ref{thm2} by using a result of Theorem \ref{thm1}.
In Section \ref{sec5}, we give some examples of the set of polynomials
which Theorems \ref{thm1} and \ref{thm2} are useful.
Throughout this paper, the term increasing function refers to a non-decreasing function.


\section{A lemma}\label{sec2}

In this section, we show the following lemma which plays an important role
in the proof of Theorem \ref{thm1}.

\begin{lemma}\label{lem1}
Let $\ell \ge1$ be an integer.
Let $f_j(x)$ $(j = 1,2,\dots,\ell)$ be integer-valued polynomials 
of degree $\ge2$ with positive leading coefficients and
$\deg f_1(x) \le \deg f_j(x)$ $(j =1,2,\dots,\ell)$.
Moreover, assume that there exists an integer $A$ such that
\begin{equation}
f_1 (x+A) \not\equiv f_j (Bx+C)+D \label{2.1.not}
\end{equation}
for any integers $j \ge 2$, $C$, $D$ and any positive rational number $B$.
Let $E$ be an arbitrary positive integer and let $G(x)$ be any real-valued function for $x \ge 1$ 
with $G(x)=o(x)$ and $G(x) \to \infty$ $(x \to \infty)$.
Then, there exist infinitely many positive integers $m$ with $m \equiv A \pmod{E}$ such that
\begin{equation*}
|f_1(m)-f_j(k)| > G(m)
\end{equation*}
for any positive integers $j$ and $k$ with $(j,k) \neq (1,m)$.
\end{lemma}


\begin{proof}
Suppose to the contrary that there exists a positive integer $m_0$ such that 
for any positive integer $m \ge m_0$ with $m \equiv A \pmod{E}$, 
there exists a positive integer pair $(j_m,k_m) \neq (1,m)$ satisfying 
\begin{equation}
|f_1(m)-f_{j_m}(k_m)| \le G(m). \label{2.1.cont}
\end{equation}
In what follows, we prove that this contradicts \eqref{2.1.not} by following five steps below.
Before Step 1, we need some preparation.

Let $d \= \deg f_1(x) \ge 2$ and 
$r_j/s_j$ $(j=1,2,\dots,\ell)$ be rational leading coefficients of $f_j(x)$,
where $r_j$ and $s_j$ are positive coprime integers. We define
\begin{equation*}
g(x) \= \max_{\substack{y \in \Z \\ 1\le y \le x}} G(y) \quad (x \ge 1).
\end{equation*}
Then, $g(x)$ $(x \ge 1)$ is an increasing function with $g(x) \to \infty$ $(x \to \infty)$ and
\begin{equation*}
g(x)=o(x).
\end{equation*}
Also, we define
\begin{align}
E_0 &\= E s_1\cdot \lcm(r_1,r_2,\dots,r_\ell), \label{2.1.E_0} \\
h(x) &\= E_0 d \ell \bigl( 2g(x)+1 \bigr). \notag
\end{align}
For every sufficiently large integer $n$, let
\begin{equation}
\cM_n \= \{ m \in \Z_{\ge 1} \mid m \equiv A \pmod{E_0},\ n-h(n) < m < n \}. \label{2.1.W}
\end{equation}
By the hypotheses, note that $f_j(x)$, $f'_j(x)$ $(j=1,2,\dots,\ell)$ are increasing for large $x$.
Here, for two real functions $F_1(x)$ and $F_2(x)$, we write $F_1(x) \asymp F_2(x)$
if there exist positive real numbers $c_1$ and $c_2$ independent of $x$ such that 
$c_1 F_2(x) \le F_1(x) \le c_2 F_2(x)$ for all large $x$.
Also for a real function $F(x)$ and a real number $K_n$ dependent on an integer $n$, 
we write $K_n \asymp F(n)$
if there exist positive real numbers $c_3$ and $c_4$ independent of $n$ such that 
$c_3 F(n) \le K_n \le c_4 F(n)$ for all large integers $n$. \\


\noindent \textbf{Step 1.}
For every sufficiently large integer $n$ and for each $j=2,3,\dots,\ell$, we define
\begin{equation*}
\cM_n(j) \= \{ m \in \cM_n \mid \text{ there exists a positive integer $k_m$ such that } 
|f_1(m)-f_j(k_m)| \le g(n) \}.
\end{equation*}
In this step, we show that there exists an integer $j_0$ $(2 \le j_0 \le \ell)$ such that
there exist infinitely many positive integers $n$ satisfying
\begin{equation}
\# \cM_n(j_0) > d \bigl( 2g(n)+1\bigr). \label{2.1.num}
\end{equation}

In what follows, let $n$ be every sufficiently large integer.
Since $h(x)=O \bigl( g(x) \bigr)=o(x)$, we get $n-h(n) \to \infty$ $(n \to \infty)$.
Then by \eqref{2.1.cont} and \eqref{2.1.W}, we see that for any integer $m \in \cM_n$, 
there exists a positive integer pair $(j_m,k_m) \neq (1,m)$ satisfying
\begin{equation}
|f_1(m)-f_{j_m}(k_m)| \le G(m) \le \max_{\substack{y \in \Z \\ 1\le y \le n}} G(y) = g(n). \label{2.1.sup}
\end{equation}
Also, we get 
\begin{equation}
f_1(x_1) \le f_1(x_2) \quad \bigl(1\le x_1\le n-h(n)-1\le x_2 \bigr). \label{2.1.x_1,2}
\end{equation}
In addition,
since $f_1(x)-f_1(x-1) \asymp x^{d-1}$ $(d \ge 2)$, $g(x)=o(x)$ and $n-h(n) \asymp n$, we obtain
\begin{equation}
f_1(x)-f_1(x-1)>2g(n) \quad \bigr(x \ge n-h(n) \bigl). \label{2.1.i}
\end{equation}
By \eqref{2.1.sup}, \eqref{2.1.x_1,2} and \eqref{2.1.i}, 
we get $j_m \neq 1$ for any integer $m \in \M_n$, and hence we have $\ell \ge 2$.
Then by \eqref{2.1.W} and \eqref{2.1.sup}, we obtain
\begin{equation*}
\# \bigcup_{j=2}^\ell \cM_n(j) = \# \cM_n \ge \frac{h(n)}{E_0}-1 > d(\ell-1) \bigl( 2g(n)+1\bigr).
\end{equation*}
Thus, for every sufficiently large integer $n$,
there exists an integer $j(n)$ $\bigl( 2 \le j(n) \le \ell \bigr)$ such that
$\# \cM_n \bigl( j(n) \bigr) > d \bigl( 2g(n)+1\bigr)$.
Therefore, we see that \eqref{2.1.num} holds. \\


\noindent \textbf{Step 2.}
In what follows, let $N$ be a sufficiently large integer satisfying 
\begin{equation}
\# \cM_N(j_0) > d \bigl( 2g(N)+1\bigr). \label{2.1.M_N}
\end{equation}
Let $t \in \cM_N(j_0)$, $u\ge1$, $v\ge1$ and $w$ be integers 
with $t+E_0v \in \cM_N(j_0)$ and $u+w \ge 1$ such that
\begin{numcases}{}
|f_1(t) - f_{j_0}(u)| \le g(N), \label{2.1.rs} \\
|f_1(t+E_0v) - f_{j_0}(u+w)| \le g(N), \label{2.1.tu}
\end{numcases}
and
\begin{equation}
(t+E_0v)-t = \min \{|m_1-m_2| \mid m_1,m_2 \in \cM_N(j_0),\ m_1 \neq m_2 \}. \label{2.1.E_0v}
\end{equation}
In this step, we show that $v$ and $w$ are bounded positive integers, 
and $\deg f_{j_0}(x) = \deg f_1(x) = d$ holds.

First if $v > 2\ell$, by \eqref{2.1.W}, \eqref{2.1.M_N} and \eqref{2.1.E_0v}, we obtain
\begin{equation*}
2E_0\ell \{d \bigl( 2g(N)+1\bigr)-1 \} < m_{\max}-m_{\min} < h(N) = E_0 d \ell \bigl( 2g(N)+1 \bigr),
\end{equation*}
where $m_{\max}$ and $m_{\min}$ are the maximum and minimum of integers $m \in \cM_N(j_0)$,
respectively. This is a contradiction for large $N$.
Thus $1\le v \le 2\ell$, and hence $v$ is a bounded positive integer.

Next, we show that $\deg f_{j_0}(x) = \deg f_1(x) = d$ holds and $w$ is a bounded positive integer.
By \eqref{2.1.i}, we get $f_1(t+E_0v) - f_1(t) > 2g(N)$.
Then by \eqref{2.1.rs} and \eqref{2.1.tu}, we obtain $w \ge 1$.
Also since $t, t+E_0v \in \cM_N(j_0)$, we get
\begin{equation}
t, t+E_0v \asymp N. \label{2.1.rt}
\end{equation}
Let $d_0 \= \deg f_{j_0}(x) \ge d$. By \eqref{2.1.rs} and \eqref{2.1.tu}, we get
\begin{equation}
u, u+w \asymp N^{\frac{d}{d_0}}. \label{2.1.su}
\end{equation}
In addition, by \eqref{2.1.rs} and \eqref{2.1.tu}, we obtain
\begin{equation*}
\left| E_0v \frac{f_1(t+E_0v) - f_1(t)}{(t+E_0v)-t} 
- w \frac{f_{j_0}(u+w) - f_{j_0}(u)}{(u+w)-u} \right| \le 2g(N).
\end{equation*}
By the mean value theorem, 
there exist positive real numbers $\a$ and $\b$ with $t<\a<t+E_0v$ and $u<\b<u+w$ such that
\begin{equation}
|E_0v f'_1(\a) - w f'_{j_0}(\b)| \le 2g(N). \label{2.1.eta_34}
\end{equation}
By \eqref{2.1.rt} and \eqref{2.1.su}, we obtain
\begin{equation*}
f'_1(\a) \asymp N^{d-1}, \quad
f'_{j_0}(\b) \asymp (N^{ \frac{d}{d_0} })^{d_0-1} = N^{ d-\frac{d}{d_0} }.
\end{equation*}
Dividing by $f'_{1}(\a)$ in \eqref{2.1.eta_34}, we get
\begin{equation}
E_0v - w \frac{f'_{j_0}(\b)}{f'_1(\a)} = O \left( \frac{g(N)}{N^{d-1}} \right). \label{2.1.Ev-w}
\end{equation}
Note that $E_0v$ is bounded and $g(N)/N^{d-1} \to 0$ $(N \to \infty)$ since $d \ge 2$.
Also, we obtain
\begin{equation*}
\frac{f'_{j_0}(\b)}{f'_1(\a)} \asymp \frac{N^{d-\frac{d}{d_0}}}{N^{d-1}} = N^{1-\frac{d}{d_0}}.
\end{equation*}
If $d_0 \ge d+1$, we get $f'_{j_0}(\b)/f'_1(\a) \to \infty$ $(N \to \infty)$.
This contradicts \eqref{2.1.Ev-w} since $w \ge 1$.
Thus $d_0 = d$, and hence $f'_{j_0}(\b)/f'_1(\a) \asymp 1$.
Then by \eqref{2.1.Ev-w}, we see that $w$ is a bounded positive integer. \\


\noindent \textbf{Step 3.}
Let $f_1(x) \= \sum_{k=0}^{d} \la_k x^k$ and $f_{j_0}(x) \= \sum_{k=0}^{d} \psi_k x^k$
with some rational constants $\la_k$, $\psi_k$ $(k=0,1,\dots,d)$.
In this step, we show
\begin{equation}
\left( \frac{\la_d}{\psi_d} \right)^\frac{1}{d} = \frac{w}{E_0v}. \label{2.1.a/b}
\end{equation}

Since $\deg f_1(x) = \deg f_{j_0}(x) = d$ and $t, u \asymp N$, by \eqref{2.1.rs} we obtain
\begin{equation*}
\la_d t^d - \psi_d u^d = O(N^{d-1}) + O\bigl( g(N) \bigr).
\end{equation*}
Dividing by $t^d$, we get $\la_d - \psi_d (u/t)^d = O(1/N)$, and hence
\begin{equation*}
\biggl( \frac{u}{t} \biggr)^d = \frac{\la_d}{\psi_d} \left( 1+ O\Bigl( \frac{1}{N} \Bigr) \right).
\end{equation*}
Here for every $\nu>0$, we have
\begin{equation}
\left( 1+O \Bigl( \frac{1}{N} \Bigr) \right)^\nu = 1+O \Bigl( \frac{1}{N} \Bigr). \label{2.1.nu}
\end{equation}
Then,
\begin{equation}
\frac{u}{t} = \left( \frac{\la_d}{\psi_d} \right)^\frac{1}{d} \left( 1+ O\Bigl( \frac{1}{N} \Bigr) \right).
\label{2.1.s/r}
\end{equation}
In addition, since $E_0v$ and $w$ are bounded, by \eqref{2.1.eta_34} we obtain
\begin{equation*}
E_0v d \la_d t^{d-1} - w d \psi_d u^{d-1} = O(N^{d-2}) + O\bigl( g(N) \bigr). 
\end{equation*}
Dividing by $d \la_d t^{d-1}$, we get
\begin{equation*}
E_0v = w \frac{\psi_d}{\la_d} \biggl( \frac{u}{t} \biggr)^{d-1}
+ O\left( \frac{1}{N} \right) + O\left( \frac{g(N)}{N^{d-1}} \right).
\end{equation*}
By \eqref{2.1.nu} and \eqref{2.1.s/r}, we obtain
\begin{align*}
E_0v &= w \frac{\psi_d}{\la_d} \left( \frac{\la_d}{\psi_d} \right)^\frac{d-1}{d} 
\left( 1+ O\Bigl( \frac{1}{N} \Bigr) \right) 
+ O\left( \frac{1}{N} \right) + O\left( \frac{g(N)}{N^{d-1}} \right) \\
&= w \left( \frac{\psi_d}{\la_d} \right)^\frac{1}{d} + O\left( \frac{g(N)}{N} \right).
\end{align*}
Since $E_0v$ is bounded, we obtain
\begin{equation*}
\frac{w}{E_0v} = \left( \frac{\la_d}{\psi_d} \right)^\frac{1}{d} + O\left( \frac{g(N)}{N} \right).
\end{equation*}
Note that $g(N)/N \to 0$ $(N \to \infty)$.
Since $E_0v$ and $w$ are bounded integers, we get \eqref{2.1.a/b}. \\


\noindent \textbf{Step 4.}
Let $(a,b)$ be every pair with positive integers $a \in \cM_N(j_0)$ and $b$ satisfying
\begin{equation}
|f_1(a) - f_{j_0}(b)| \le g(N). \label{2.1.yz} \\
\end{equation}
In this step, we show that 
there exists a rational number $\xi_0$ independent of $a$ and $b$ such that
\begin{equation}
b = \left( \frac{\la_d}{\psi_d} \right)^\frac{1}{d} \ a + \xi_0. \label{2.1.xi_0}
\end{equation}

Since $a \asymp N$, we get $b \asymp N$.
Similarly to \eqref{2.1.s/r}, by \eqref{2.1.yz} we obtain
\begin{equation*}
\frac{b}{a} = \left( \frac{\la_d}{\psi_d} \right)^\frac{1}{d} \left( 1+ O\Bigl( \frac{1}{N} \Bigr) \right).
\end{equation*}
Multiplying by $a$, we get
\begin{equation}
b = \left( \frac{\la_d}{\psi_d} \right)^\frac{1}{d} a + O(1). \label{2.1.z}
\end{equation}
Let 
\begin{equation}
\xi_{a,b} \= b - \left( \frac{\la_d}{\psi_d} \right)^\frac{1}{d} a \label{2.1.xi_ab}
\end{equation}
for each pair $(a,b)$ satisfying \eqref{2.1.yz}.
By \eqref{2.1.z}, we see that $\xi_{a,b}$ is bounded.
Then, substituting $b = (\la_d/\psi_d)^{1/d} \cdot a + \xi_{a,b}$ in \eqref{2.1.yz}, we obtain
\begin{align*}
& \la_d a^d + \la_{d-1} a^{d-1}
- \psi_d \left\{ \left( \frac{\la_d}{\psi_d} \right)^\frac{1}{d} a \right\}^d 
- \psi_d d \left\{ \left( \frac{\la_d}{\psi_d} \right)^\frac{1}{d} a \right\}^{d-1} \xi_{a,b} 
- \psi_{d-1} \left\{ \left( \frac{\la_d}{\psi_d} \right)^\frac{1}{d} a \right\}^{d-1} \\
= \ & O(N^{d-2}) +O\bigl( g(N) \bigr).
\end{align*}
Dividing by $a^{d-1}$, we get
\begin{equation*}
\la_{d-1} - d \psi_d \left( \frac{\la_d}{\psi_d} \right)^\frac{d-1}{d} \xi_{a,b} 
- \psi_{d-1} \left( \frac{\la_d}{\psi_d} \right)^\frac{d-1}{d} = O\left( \frac{g(N)}{N} \right).
\end{equation*}
Then, we get
\begin{equation*}
\xi_{a,b} = \frac{\la_{d-1}}{d \la_d} \left( \frac{\la_d}{\psi_d} \right)^\frac{1}{d} 
- \frac{\psi_{d-1}}{d \psi_d} + O\left( \frac{g(N)}{N} \right).
\end{equation*}
Note that $g(N)/N \to 0$ $(N \to \infty)$.
Since $E_0v$ is a bounded integer, by \eqref{2.1.a/b} and \eqref{2.1.xi_ab},
$\xi_{a,b}$ is a rational number which the denominator is bounded, and hence we obtain
\begin{equation*}
\xi_{a,b} = \frac{\la_{d-1}}{d \la_d} \left( \frac{\la_d}{\psi_d} \right)^\frac{1}{d} 
- \frac{\psi_{d-1}}{d \psi_d}.
\end{equation*}
Therefore, for each pair $(a,b)$ satisfying \eqref{2.1.yz}, 
$\xi_{a,b}$ is a rational number independent of $a$ and $b$.
Thus by \eqref{2.1.xi_ab}, we see that \eqref{2.1.xi_0} holds. \\


\noindent \textbf{Step 5.}
Finally, we show that \eqref{2.1.cont} contradicts \eqref{2.1.not}. 
By \eqref{2.1.M_N}, \eqref{2.1.yz} and \eqref{2.1.xi_0}, 
there exists an integer $D_0$ with $|D_0| \le g(N)$ such that the equation
\begin{equation*}
f_1(x) = f_{j_0} \left( \left( \frac{\la_d}{\psi_d} \right)^\frac{1}{d} x+\xi_0 \right) + D_0
\end{equation*}
has at least $d+1$ solutions. 
Since $\deg f_1(x)= \deg f_{j_0}(x) = d$, we obtain
\begin{equation*}
f_1(x) \equiv f_{j_0} \left( \left( \frac{\la_d}{\psi_d} \right)^\frac{1}{d} x+\xi_0 \right) + D_0.
\end{equation*}
Here by \eqref{2.1.rs} and \eqref{2.1.xi_0}, we get
\begin{equation}
u = \left( \frac{\la_d}{\psi_d} \right)^\frac{1}{d} t +\xi_0. \label{2.1.s}
\end{equation}
Replacing $x$ with $x+t$, by \eqref{2.1.a/b} and \eqref{2.1.s}, we obtain
\begin{equation}
f_1(x+t) \equiv f_{j_0} \left( \frac{w}{E_0v} \ x+u \right) + D_0. \label{2.1.t}
\end{equation}
Now comparing the leading coefficients, we get
$r_1/s_1= (r_{j_0}/s_{j_0}) \{ w/(E_0v) \}^d$, namely,
\begin{equation*}
r_1 s_{j_0} E_0^d = r_{j_0} s_1 \left( \frac{w}{v} \right)^d.
\end{equation*}
By \eqref{2.1.E_0}, we get $r_{j_0} s_1 \mid E_0$. Then, $B_0 \= w/v$ is a positive integer
since $v$ and $w$ are positive integers.
Also since $t \in \cM_N(j_0)$, there exists a positive integer $T$ such that $t=E_0 T +A$.
By \eqref{2.1.t}, replacing $x$ with $x-E_0 T$, we obtain
\begin{equation*}
f_1(x+A) \equiv f_{j_0} \left( \frac{B_0}{E_0} x + (u-B_0 T) \right) +D_0.
\end{equation*}
This contradicts \eqref{2.1.not}. 
Therefore, the proof of Lemma \ref{lem1} is completed.
\end{proof}


\section{Proof of Theorem~\ref{thm1}}\label{sec3}

\begin{proof}
Suppose to the contrary that 
the numbers \eqref{1.1.ind} in Theorem \ref{thm1} are linearly dependent over $\Q(q)$. 
Then, there exist algebraic integers $b_j$ $(j=0,1,\dots,\ell)$ in $\Q(q)$ not all zero such that 
\begin{equation}
b_0 +\sum_{j=1}^{\ell} b_j \sum_{n=1}^{\infty} \frac{a_j (n)}{q^{f_j (n)}} =0. \label{1.1.ga}
\end{equation}
Let 
\begin{align*}
\cR &\= \{ j \in \Z_{\ge1} \mid b_j \neq 0,\ j=1,2,\dots,\ell \}, \\
\cS &\= \bigl\{ j \in \cR \mid \deg f_j(x) = \min_{k \in \cR} \{ \deg f_k(x) \} \bigr\}.
\end{align*}
Since $\cR \neq \emptyset$, we also obtain $\cS \neq \emptyset$.
By the conditions \eqref{i}, \eqref{ii} in Theorem \ref{thm1}, 
for the non-empty subset $\cS$ of $\{1,2,\dots,\ell \}$, 
there exist integers $i \in \cS$ and $A$ satisfying the following two conditions:
\begin{enumerate}
\setlength{\parskip}{0pt}
\setlength{\itemsep}{0pt}
\item[$\text{\rm(I)}$]
$f_i (x+A) \not\equiv f_j (Bx+C)+D$ for any integers 
$j \in \cS \setminus \{i\}$, $C$, $D$ and any positive rational number $B$.
\item[$\text{\rm(II)}$]
There exists a positive integer $E$ such that $\displaystyle \liminf_{n \to \infty} |a_i (En+A)| \neq 0$.
\end{enumerate}
Without loss of generality, we may assume $i=1$. Then, we get
\begin{equation}
f_1 (x+A) \not\equiv f_j (Bx+C)+D \label{1.1.not}
\end{equation}
for any integers $j \in \cR \setminus \{1\}$, $C$, $D$
and any positive rational number $B$, and
\begin{equation}
\liminf_{n \to \infty} |a_1(En+A)| \neq 0. \label{1.1.neq}
\end{equation}

Next, we apply Lemma \ref{lem1}. We define
\begin{equation}
h(x) \= 1+ \max_{\substack{n \in \Z \\ 1 \le n \le x}} \left\{
\sum_{j \in \cR} \sum_{\sigma} \log_{q} (1+|a_j(n)^\sigma|) \right\} \quad (x \ge 1), \label{1.1.h}
\end{equation}
where $\sigma$ runs through all embeddings of $\Q(q) \to \C$.
Also, we define
\begin{align}
H(x) &\= \min_{\substack{y \\ x \le y}} \frac{y}{h(y)} \quad (x \ge 1), \label{1.1.H} \\
G(x) &\= x \cdot H(x)^{-\frac{1}{4}} \quad (x \ge 1). \label{1.1.G}
\end{align}
By \eqref{1.1.h}, we obtain $h(x)=o(x)$ since $\log(1+|a_j (n)^\sigma|)=o(n)$
for every $j \in \cR$ and every embedding $\sigma \colon \Q(q) \to \C$.
Then by \eqref{1.1.H}, we get
\begin{equation}
H(x) \to \infty \quad (x \to \infty), \label{1.1.H_inf}
\end{equation}
and hence by \eqref{1.1.G}, we obtain
\begin{equation}
G(x)=o(x). \label{1.1.G_o}
\end{equation}
Also by \eqref{1.1.h}, we get $h(x) \ge 1$ for $x \ge 1$. 
Then by \eqref{1.1.H}, we get
\begin{equation}
0 < H(x) \le \frac{x}{h(x)} \le x \quad (x \ge 1), \label{1.1.H_x/h}
\end{equation}
and hence by \eqref{1.1.G}, we obtain
\begin{equation}
G(x) \ge x^{\frac{3}{4}} \to \infty \quad (x \to \infty). \label{1.1.G_inf}
\end{equation}
By \eqref{1.1.not}, \eqref{1.1.G_o} and \eqref{1.1.G_inf}, Lemma \ref{lem1} implies that 
there exist infinitely many positive integers $m$ with $m \equiv A \pmod{E}$ such that
\begin{equation}
|f_1(m)-f_j(k)| > G(m) \label{1.1.main}
\end{equation}
for any positive integers $j \in \cR$ and $k$ with $(j,k) \neq (1,m)$.

Let $m$ be a sufficiently large integer satisfying \eqref{1.1.main}. 
By the hypotheses, note that $f_j(x)$ $(j \in \cR)$ are increasing for large $x$.
We define
\begin{align}
n_{m,j} &\= \max \{n \in \Z_{\ge 1} \mid f_j(n) < f_1(m)-G(m) \} \quad (j \in \cR), \label{1.1.n_j,m} \\
N_{m,j} &\= \min \{n \in \Z_{\ge 1} \mid f_j(n) > f_1(m)+G(m) \} \quad (j \in \cR), \label{1.1.N_j,m} \\
K_m &\= \max_{j \in \cR} \{ f_j(n_{m,j}) \}. \label{1.1.K_m}
\end{align} 
Multiplying \eqref{1.1.ga} by $q^{K_m}$,
by \eqref{1.1.main}, \eqref{1.1.n_j,m} and \eqref{1.1.N_j,m}, we obtain
\begin{equation*}
b_0 q^{K_m} + \sum_{j \in \cR} b_j \sum_{n=1}^{n_{m,j}} a_j (n) q^{K_m-f_j (n)} 
+ b_1\ \frac{a_1(m)}{q^{f_1(m) - K_m}}
+ \sum_{j \in \cR} b_j \sum_{n=N_{m,j}}^{\infty} \frac{a_j (n)}{q^{f_j (n) - K_m}} = 0.
\end{equation*}
Let 
\begin{equation}
P_m(X) \= -b_0 X^{K_m} - \sum_{j \in \cR} b_j \sum_{n=1}^{n_{m,j}} a_j (n) X^{K_m-f_j (n)}. \label{1.1.P}
\end{equation}
Then,
\begin{equation}
P_m(q) = \frac{a_1(m)}{q^{f_1(m) - K_m}} \left\{b_1 + \sum_{j \in \cR} b_j
\sum_{n=0}^{\infty} \frac{a_j (N_{m,j}+n)}{a_1(m) q^{f_j (N_{m,j}+n) - f_1(m)}} \right\}. \label{1.1.q_ga}
\end{equation}

Now, we evaluate \eqref{1.1.q_ga}.
By $\deg f_1 (x) \le \deg f_j (x)$ $(j \in \cR)$ and \eqref{1.1.G_o},
there exists a positive integer $c_0$ independent of $m$
such that $f_j (c_0m) > f_1(m)+G(m)$ $(j \in \cR)$.
Then by \eqref{1.1.H_inf}, \eqref{1.1.n_j,m} and \eqref{1.1.N_j,m}, we get
\begin{equation}
n_{m,j} < N_{m,j} \le c_0m < m \cdot H(m)^{\frac{1}{2}} \quad (j \in \cR). \label{1.1.base}
\end{equation}
By \eqref{1.1.H_x/h}, we have $h(x) \le x/H(x)$ $(x \ge 1)$.
In addition, by \eqref{1.1.h} and \eqref{1.1.H}, $h(x)$ and $H(x)$ $(x \ge 1)$ are increasing functions.
Then by \eqref{1.1.H_inf} and \eqref{1.1.base}, for every integers $j \in \cR$ and $n\ge0$, we obtain
\begin{align*}
h(N_{m,j} +n) &\le h \bigl( m \cdot H(m)^{\frac{1}{2}} +n \bigr) \le 
\frac{m \cdot H(m)^{\frac{1}{2}} +n}{H \bigl( m \cdot H(m)^{\frac{1}{2}} +n \bigr) } \le
\frac{m \cdot H(m)^{\frac{1}{2}} +n}{H(m)} \\
&\le m \cdot H(m)^{-\frac{1}{2}} +n.
\end{align*}
Then by \eqref{1.1.G}, for every integers $j \in \cR$ and $n\ge0$, we obtain
\begin{equation}
h(N_{m,j} +n) \le G(m) \cdot H(m)^{-\frac{1}{4}} +n. \label{1.1.h(n)}
\end{equation}
Let $c_1,c_2,\dots$ be some positive constants independent of $m$.
By \eqref{1.1.neq}, we obtain $|a_1(m)| > c_1$.
Also by \eqref{1.1.h}, for every $j \in \cR$ and every embedding $\sigma \colon \Q(q) \to \C$,
we get
\begin{equation*}
|a_j (n)^\sigma| \le q^{h(n)} \quad (n \ge 1).
\end{equation*}
Since $h(x)$ $(x\ge1)$ is an increasing function, by \eqref{1.1.N_j,m}, we get $h(m) \le h(N_{m,1})$.
In addition, note that $G(x) \to \infty$ and $H(x) \to \infty$ $(x \to \infty)$.
By \eqref{1.1.n_j,m}, \eqref{1.1.K_m} and \eqref{1.1.h(n)}, we get
\begin{equation}
0 < \left| \frac{a_1(m)}{q^{f_1(m) - K_m}} \right| 
\le \frac{q^{h(m)}}{q^{G(m)}} 
\le \frac{q^{h(N_{m,1})}}{q^{G(m)}}
\le q^{-G(m) \left( 1-H(m)^{-1/4} \right)} \to 0 \quad (m \to \infty). \label{1.1.q_ga_1}
\end{equation}
Here by \eqref{1.1.N_j,m}, if $m$ is large, $N_{m,j}$ is also large for every $j \in \cR$.
Then, $f_j (x+1) - f_j (x) \ge 2$ for $x \ge N_{m,j}$ $(j \in \cR)$.
Thus by \eqref{1.1.N_j,m} and \eqref{1.1.h(n)}, for each $j \in \cR$, we obtain
\begin{align}
\left| \sum_{n=0}^{\infty} \frac{a_j (N_{m,j}+n)}{a_1(m) q^{f_j (N_{m,j}+n) - f_1(m)}} \right| 
&\le c_2\ \frac{1}{q^{f_j (N_{m,j}) - f_1(m)}} 
  \sum_{n=0}^{\infty} \frac{q^{h(N_{m,j}+n)}}{q^{f_j (N_{m,j}+n)-f_j (N_{m,j})}} \notag \\
&\le c_2\ \frac{q^{G(m) \cdot H(m)^{-1/4}}}{q^{G(m)}} \sum_{n=0}^{\infty} \frac{q^n}{q^{2n}} \notag \\
&= c_2 \ \frac{q}{q-1} \ q^{-G(m) \left (1-H(m)^{-1/4} \right)} \to 0 \quad (m \to \infty).
  \label{1.1.q_ga_2}
\end{align}
In addition, note that $b_j$ $(j \in \cR)$ are non-zero constants of algebraic integers in $\Q(q)$.
By \eqref{1.1.q_ga}, \eqref{1.1.q_ga_1} and \eqref{1.1.q_ga_2}, we obtain
\begin{equation}
0 < |P_m(q)| \le c_3 q^{-G(m) \left( 1-H(m)^{-1/4} \right)} \to 0 \quad (m \to \infty). \label{1.1.P(q)}
\end{equation}

Finally, we evaluate the norm over $\Q(q)/\Q$ of the number $P_m(q)$.
Note that $h(x)$ $(x \ge 1)$ is an increasing function.
Then, for every embedding $\sigma \colon \Q(q) \to \C$ with $q^\sigma \neq q$, 
by $|q^\sigma| \le 1$, \eqref{1.1.K_m}, \eqref{1.1.P}, \eqref{1.1.base} and \eqref{1.1.h(n)}, we get
\begin{align} 
|P_m(q)^\sigma| &\le |b_0^\sigma| \cdot |q^\sigma|^{K_m} + \sum_{j \in \cR} |b_j^\sigma| 
\sum_{n=1}^{n_{m,j}} |a_j (n)^\sigma| \cdot |q^\sigma|^{K_m-f_j (n)} \notag \\
&\le c_4 \sum_{j \in \cR} \sum_{n=1}^{n_{m,j}} |a_j(n)^\sigma|
\le c_4 \sum_{j \in \cR} n_{m,j} q^{h(N_{m,j})} 
\le c_5 m q^{G(m) \cdot H(m)^{-1/4}}. \label{1.1.P(q)^s}
\end{align}
By \eqref{1.1.K_m}, \eqref{1.1.P} and \eqref{1.1.P(q)}, 
we see that $P_m(q)$ is a non-zero algebraic integer in $\Q(q)$.
Thus by \eqref{1.1.H_inf}, \eqref{1.1.G_inf}, \eqref{1.1.P(q)} and \eqref{1.1.P(q)^s}, we obtain
\begin{align*}
1 \le |N_{\Q(q)/\Q} (P_m(q))| 
&\le c_3 q^{-G(m) \left( 1- H(m)^{-1/4} \right)} 
\left( c_5 m q^{G(m) \cdot H(m)^{-1/4}} \right)^{[\Q(q):\Q]-1} \\
&\le c_6 m^{[\Q(q):\Q]-1} q^{-m^{3/4} \left( 1-[\Q(q):\Q] \cdot H(m)^{-1/4} \right)} 
\to 0 \quad (m \to \infty).
\end{align*}
This is a contradiction for the sufficiently large integer $m$. 
Therefore, the proof of Theorem \ref{thm1} is completed.
\end{proof}


\section{Proof of Theorem \ref{thm2}}\label{sec4}

\begin{proof}
Suppose that there exist algebraic integers $b_0$, $b_{i,j}$ $\bigl( (i,j) \in \cU \bigr)$ in $\Q(q)$ 
such that
\begin{equation*}
b_0 + \sum_{i=1}^m \sum_{j=1}^{\ell_i} b_{i,j} \sum_{n=1}^\infty \frac{ P_{i,j}(n) }{ q^{f_{i,j}(n)} } = 0.
\end{equation*}
For each $(i,j) \in \cU$, we get
\begin{equation*}
\sum_{n=1}^\infty \frac{ P_{i,j}(n) }{ q^{f_{i,j}(n)} } 
= \sum_{n=0}^\infty \frac{ \tilde{p}_{i,j}(n) }{ q^{f_{i,j} \bigl( \frac{n+C_{i,j}}{B_{i,j}} \bigr)} } + s_{i,j}
= \sum_{n=0}^\infty \frac{ q^{D_{i,j}} \cdot \tilde{p}_{i,j}(n) }{ q^{g_i(n)} } +s_{i,j},
\end{equation*}
where $s_{i,j} \in \Q(q)$ is some constant.
Let $d$ be the minimum of $D_{i,j}$ $\bigl( (i,j) \in \cU \bigr)$.
Also, for each $i=1,2,\dots,m$, let $K_i$ be the least common multiple of $B_{i,j}$ $(j=1,2,\dots,\ell_i)$.
Then,
\begin{equation*}
s_0 + q^d \sum_{i=1}^m \sum_{r=0}^{K_i-1} \sum_{n=0}^\infty \frac{ \displaystyle \sum_{j=1}^{\ell_i} 
b_{i,j} q^{D_{i,j}-d} \cdot \tilde{p}_{i,j}(K_i n + r) }{ q^{g_i(K_i n + r)} } = 0,
\end{equation*}
where $s_0 \= b_0 + \sum_{(i,j) \in \cU} b_{i,j} s_{i,j} \in \Q(q)$.
Note that $K_i/B_{i,j}$ $\bigl( (i,j) \in \cU \bigr)$ are positive integers. 
For each $(i,j) \in \cU$, since 
\begin{equation}
\tilde{p}_{i,j}(K_i n + r) =
\begin{dcases*}
P_{i,j} \left( \frac{K_i}{B_{i,j}} n + \frac{r+C_{i,j}}{B_{i,j}} \right) & if $\ r \equiv -C_{i,j} \pmod{B_{i,j}}$, \\
0 & otherwise,
\end{dcases*} \label{4.1.b_i,j}
\end{equation}
we see that for each $i=1,2,\dots,m$ and $r=0,1,\dots,K_i-1$, 
\begin{equation}
\sum_{j=1}^{\ell_i} b_{i,j} q^{D_{i,j}-d} \cdot \tilde{p}_{i,j}(K_i n + r) \label{4.1.poly}
\end{equation}
is a polynomial of $n$ with algebraic integer coefficients in $\Q(q)$.
Now, we define
\begin{equation*}
\cV \= \{ (i,r) \in \Z^2 \mid 1\le i \le m, \ 0 \le r \le K_i-1,
\text{and \eqref{4.1.poly} is a non-zero polynomial of $n$} \}.
\end{equation*}
We show $\cV = \emptyset$.

Suppose to the contrary $\cV \neq \emptyset$. 
For each $(i,r) \in \cV$, let $h_{i,r} (x) \= g_i (K_i x + r)$. Then, 
\begin{equation}
s_0 + q^d \sum_{(i,r) \in \cV} \sum_{n=0}^\infty \frac{ \displaystyle \sum_{j=1}^{\ell_i} 
b_{i,j} q^{D_{i,j}-d} \cdot \tilde{p}_{i,j}(K_i n + r) }{ q^{h_{i,r}(n)} } = 0. \label{4.1.sup}
\end{equation}
We apply Theorem \ref{thm1}. By the definition of $\cV$, we obtain
\begin{equation*}
\liminf_{n \to \infty} \left| \sum_{j=1}^{\ell_i} b_{i,j} q^{D_{i,j}-d} \cdot \tilde{p}_{i,j}(K_i n + r) \right| 
\neq 0 \quad \bigl( (i,r) \in \cV \bigr).
\end{equation*}
By \eqref{4.1.b_i,j}, if $(i,r) \in \cV$, there exists an integer $j_0$ $(1\le j_0 \le \ell_i)$ such that 
$(r+C_{i,j_0})/B_{i,j_0}$ is an integer, and hence
\begin{equation*}
h_{i,r}(x) \equiv f_{i,j_0} \left( \frac{K_i}{B_{i,j_0}} x + \frac{r+C_{i,j_0}}{B_{i,j_0}} \right) + D_{i,j_0}
\end{equation*}
is an integer-valued polynomial of degree $\ge2$ with a positive leading coefficient.
Also, for each $i=1,2,\dots,m$, we obtain
\begin{equation}
h_{i,r_1}(x+A) \not\equiv h_{i,r_2}(Bx+C) + D \label{4.1.F_r}
\end{equation}
for any integers $0 \le r_1,\ r_2 \le K_i-1$ $(r_1\neq r_2)$, $A$, $C$, $D$ 
and any positive rational number $B$.
Indeed, suppose to the contrary that there exist integers
$i$ $(1\le i \le m)$, $r_1$, $r_2$ $(0 \le r_1,\ r_2 \le K_i-1,$ $r_1\neq r_2)$, $A$, $C$, $D$ 
and a positive rational number $B$ such that
\begin{equation*}
h_{i,r_1}(x+A) \equiv h_{i,r_2}(Bx+C) + D.
\end{equation*}
Namely, 
\begin{equation*}
g_i(K_i x + AK_i+r_1) \equiv g_i(BK_i x + CK_i+r_2) + D.
\end{equation*}
Then, we get $B=1$, $AK_i+r_1 = CK_i+r_2$ and $D=0$.
In particular, we get $r_1-r_2=K_i(C-A)$, and this contradicts 
$0 \le r_1,\ r_2 \le K_i-1$ and $r_1\neq r_2$.
Thus, we see that \eqref{4.1.F_r} holds.
In addition, we have
\begin{equation*}
g_{i_1}(x) \not\sim g_{i_2}(x) \quad (1 \le i_1 < i_2 \le m).
\end{equation*}
Hence, we obtain
\begin{equation*}
h_{i_1,r_1}(x+A) \not\equiv h_{i_2,r_2}(Bx+C) + D
\end{equation*}
for any $(i_1,r_1) \neq (i_2,r_2)$, any integers $A$, $C$, $D$ and any positive rational number $B$.
Therefore, by Theorem \ref{thm1}, the numbers
\begin{equation*}
1, \qquad \sum_{n=0}^\infty \frac{ \displaystyle \sum_{j=1}^{\ell_i} 
b_{i,j} q^{D_{i,j}-d} \cdot \tilde{p}_{i,j}(K_i n + r)}{ q^{h_{i,r}(n)} } \quad \bigl( (i,r) \in \cV \bigr)
\end{equation*}
are linearly independent over $\Q(q)$, and this contradicts \eqref{4.1.sup}.
Thus $\cV = \emptyset$, and hence for each $i=1,2,\dots,m$, we obtain
\begin{equation*}
\sum_{j=1}^{\ell_i} b_{i,j} q^{D_{i,j}-d} \cdot \tilde{p}_{i,j}(K_i n+r) = 0
\end{equation*}
for every integers $0 \le r \le K_i-1$ and $n \ge 0$. 
Therefore, if the numbers \eqref{1.2.ind} in Theorem \ref{thm2} are linearly dependent over $\Q(q)$, 
then, for each $i=1,2,\dots,m$, the sequences 
$\{ \tilde{p}_{i,j}(n) \}_{n \ge 0}$ $(j=1,2,\dots,\ell_i)$ are linearly dependent over $\Q(q)$.
Conversely, assuming for each $i=1,2,\dots,m$ that 
$\{ \tilde{p}_{i,j}(n) \}_{n \ge 0}$ $(j=1,2,\dots,\ell_i)$ are linearly dependent over $\Q(q)$,
then we can show by \eqref{4.1.sup} that 
the numbers \eqref{1.2.ind} in Theorem \ref{thm2} are linearly dependent over $\Q(q)$.
The proof of Theorem \ref{thm2} is completed.
\end{proof}


\section{Examples}\label{sec5}

\subsection{Examples of Theorem \ref{thm1}}

First, we give some examples of the set of polynomials
which satisfy the condition \eqref{i} in Theorem \ref{thm1}.
We give the following lemma to prove Example \ref{exm2} in Section \ref{sec1}.

\begin{lemma}\label{lem2}
Let $g(x)$ be any integer-valued polynomial of degree $d \ge 1$ with a positive leading coefficient. Let
\begin{equation}
f_{i,j} (x) \= i g(x)^j + h_{i,j}(x) 
\quad \bigl(i=1,2,\dots,\ell_1,\ j=1,2,\dots,\ell_2 \text{ with } j \ge \frac{2}{d} \bigr), \label{5.1.poly}
\end{equation}
where $h_{i,j} (x)$ $(i=1,2,\dots,\ell_1,\ j=1,2,\dots,\ell_2 \text{ with } j \ge 2/d)$ 
are any integer-valued polynomials of degree $\le dj-2$.
Then, the polynomials \eqref{5.1.poly} satisfy the condition \eqref{i} in Theorem \ref{thm1}.
\end{lemma}

\begin{proof}
Let 
\begin{equation*}
A_0 \= -\left\lfloor \frac{u}{td} \right\rfloor -1,
\end{equation*}
where $t>0$ and $u$ are coefficients of $x^d$ and $x^{d-1}$ in $g(x)$, respectively.
Let $\cS$ be any non-empty subset of 
$\{ (i,j) \mid i=1,2,\dots,\ell_1,\ j=1,2,\dots,\ell_2 \text{ with } j \ge 2/d \}$
and $(i_1, j_1) \in \cS$ be an element satisfying $i_1=\min \{i \mid (i,j) \in \cS \}$.
Suppose to the contrary that
there exist $(i_2,j_2) \in \cS \setminus \{ (i_1,j_1) \}$, a positive rational number $B$
and integers $C$, $D$ such that
\begin{equation*}
f_{i_1,j_1} (x+A_0) \equiv f_{i_2,j_2} (Bx+C)+D.
\end{equation*}
Namely,
\begin{equation*}
i_1g(x+A_0)^{j_1} + h_{i_1,j_1}(x+A_0) \equiv i_2 g(Bx+C)^{j_2} + h_{i_2,j_2}(Bx+C) + D.
\end{equation*}
Then, we get $j_1=j_2$, and hence $i_1<i_2$.
By comparing the coefficients of $x^{d j_1}$ and $x^{d j_1-1}$, we obtain
\begin{numcases}{}
i_1 t^{j_1} = i_2 (t B^d)^{j_1}, \label{5.1.coe1} \\
i_1 j_1 t^{j_1-1} (td A_0+u) = i_2 j_1 (t B^d)^{j_1-1} B^{d-1} (tdC + u), \label{5.1.coe2}
\end{numcases}
respectively.
By \eqref{5.1.coe1}, we obtain $i_1=i_2 (B^d)^{j_1}$. Since $i_1<i_2$, we get $0<B<1$. 
In addition, substituting in \eqref{5.1.coe2}, we obtain $B(tdA_0+u) = tdC+u$.
Then,
\begin{align*}
C &= BA_0 - (1-B) \frac{u}{td} \\
&= B \left( -\left\lfloor \frac{u}{td} \right\rfloor -1 \right) 
    - (1-B) \left( \left\lfloor \frac{u}{td} \right\rfloor + \left\{ \frac{u}{td} \right\} \right) \\
&= - \left\lfloor \frac{u}{td} \right\rfloor -1 + (1-B) \left( 1- \left\{ \frac{u}{td} \right\} \right),
\end{align*}
where $\{x\} \= x - \lfloor x \rfloor$ denotes the fractional part of $x$.
Note that 
\begin{equation*}
0 < (1-B) \left( 1- \left\{ \frac{u}{td} \right\} \right) <1.
\end{equation*}
This is a contradiction since $C$ is an integer.
Thus, we obtain
\begin{equation*}
f_{i_1,j_1} (x+A_0) \not\equiv f_{i_2,j_2} (Bx+C)+D
\end{equation*}
for any $(i_2,j_2) \in \cS \setminus \{ (i_1,j_1) \}$, any positive rational number $B$
and any integers $C$, $D$.
Therefore, the proof of Lemma \ref{lem2} is completed.
\end{proof}


By Corollary \ref{cor1} in Section \ref{sec1}, we get the following corollary.

\begin{cor}\label{cor3}
Let $q$ be a Pisot or Salem number.
Let $f_{i,j}(x)$ $(i=1,2,\dots,\ell_1, \ j=1,2,\dots,\ell_2$ with $j \ge 2/d)$
be polynomials defined by \eqref{5.1.poly} in Lemma \ref{lem2}.
Then, for any sequences $\{a_{i,j} (n)\}_{n\ge1}$ 
$(i\!=1,2,\dots,\ell_1, \ j\!=1,2,\dots,\ell_2 \text{ with } j \ge 2/d)$
of non-zero integers with $\log(1+|a_{i,j} (n)|)=o(n)$, the numbers
\begin{equation*}
1,\qquad \sum_{n=1}^{\infty} \frac{a_{i,j} (n)}{q^{f_{i,j} (n)}} 
\quad \bigl( i=1,2,\dots,\ell_1,\ j=1,2,\dots,\ell_2 \text{ with } j \ge \frac{2}{d} \bigr)
\end{equation*}
are linearly independent over $\Q(q)$.
\end{cor}

In Lemma \ref{lem2}, if $g(x) \equiv x$, $h_{i,j}(x) \equiv 0$ and $\ell_1=\ell_2=\ell$, 
we get Example \ref{exm2} in Section \ref{sec1}.
Also in Lemma \ref{lem2}, if $g(x) \equiv x(x+1)$ or $x^2$, 
$h_{i,j}(x) \equiv 0$, $\ell_1 = \ell$ and $\ell_2 = 1$, we get linear independence results
for the values of the Jacobi theta constants defined by \eqref{1.1.Jacobi} in Section \ref{sec1}.

\begin{example}\label{exm4}
Let $q$ be a Pisot or Salem number. Then, for any integer $\ell \ge 1$, the numbers
\begin{equation*}
1, \qquad q^{\frac{i}{4}} \th_2 (1/q^i) \quad (i=1,2,\dots,\ell)
\end{equation*}
are linearly independent over $\Q(q)$. So are the numbers
\begin{equation*}
1, \qquad \th_3 (1/q^i) \quad (i=1,2,\dots,\ell),
\end{equation*}
and the numbers
\begin{equation*}
1, \qquad \th_4 (1/q^i) \quad (i=1,2,\dots,\ell).
\end{equation*}
\end{example}

Note that the numbers $1$, $q^{i/4}\th_2 (1/q^i)$, $\th_3 (1/q^i)$ $(i=1,2,\dots,\ell)$
are linearly dependent over $\Q(q)$ if $\ell \ge 4$.
For example, $\th_3(1/q) = \th_2(1/q^4) + \th_3(1/q^4)$ since
\begin{equation*}
1+ 2\sum_{n=1}^{\infty} \frac{1}{q^{n^2}} = 2\sum_{n=0}^{\infty} \frac{1}{q^{(2n+1)^2}} +
\left(1+ 2\sum_{n=1}^{\infty} \frac{1}{q^{(2n)^2}} \right).
\end{equation*}
Also, it is known that for any integer $\ell \ge1$ and any algebraic number $\a$ $(0<|\a|<1)$,
any two numbers among the numbers $\th_m (\a^i)$ $(m=2,3,4, \ i=1,2,\dots,\ell)$ 
are algebraically independent over $\Q$, 
while any three are not (cf. \cite{Tachiya1, Tachiya2, Tachiya3}).


\subsection{An example of Theorem \ref{thm2}}

\begin{example}\label{exm5}
Let $q$ be a Pisot or Salem number.
Let $P_1(x)$, $P_2(x)$ and $P_3(x)$ 
be non-zero polynomials with algebraic integer coefficients in $\Q(q)$.
Then, the four numbers
\begin{equation*}
1, \quad \sum_{n=1}^\infty \frac{P_1(n)}{q^{n^2}},
\quad \sum_{n=1}^\infty \frac{P_2(n)}{q^{(2n-1)^2}}, 
\quad \sum_{n=1}^\infty \frac{P_3(n)}{q^{(2n)^2}}
\end{equation*}
are linearly dependent over $\Q(q)$ if and only if 
there exist algebraic numbers $c_1 \neq 0$ and $c_2 \neq 0$ in $\Q(q)$ such that
\begin{equation}
P_2(x) \equiv c_1 P_1(2x-1), \quad P_3(x) \equiv c_2 P_1(2x). \label{5.2.con}
\end{equation}
\end{example}

\begin{proof}
We apply the criterion given by Theorem \ref{thm2}.
Let $f_1(x) \= x^2$, $f_2(x) \= (2x-1)^2$, $f_3(x) \= (2x)^2$ and $g(x) \= x^2$. 
Then, we get $f_1 \sim f_2 \sim f_3$ and 
\begin{equation*}
g(x) \equiv f_1(x) \equiv f_2 \Bigl( \frac{x+1}{2} \Bigr) \equiv f_3 \Bigl( \frac{x}{2} \Bigr).
\end{equation*}
Let $\{ \tilde{p}_j (n) \}_{n \ge 0}$ $(j=1,2,3)$ be sequences such that
\begin{align*}
\tilde{p}_1(n) &\= P_1(n) \quad \text{for every integer $n$,} \\
\tilde{p}_2(n) &\=
\begin{dcases*}
P_2 \Bigl( \frac{n+1}{2} \Bigr) & if $n$ is odd, \\
0 & if $n$ is even, \\ 
\end{dcases*} \\
\tilde{p}_3(n) &\=
\begin{dcases*}
P_3 \Bigl( \frac{n}{2} \Bigr) & if $n$ is even, \\
0 & if $n$ is odd.
\end{dcases*}
\end{align*}
Assume that there exist algebraic numbers $k_1$, $k_2$ and $k_3$ in $\Q(q)$ such that
\begin{equation*}
k_1 \cdot \tilde{p}_1(n) + k_2 \cdot \tilde{p}_2(n) + k_3 \cdot \tilde{p}_3(n) = 0
\end{equation*}
for every integer $n \ge 0$.
Then, we obtain 
\begin{align*}
\begin{dcases*}
k_1 \cdot P_1(n) + k_2 \cdot P_2 \Bigl( \frac{n+1}{2} \Bigr) = 0 & if $n$ is odd, \\
k_1 \cdot P_1(n) + k_3 \cdot P_3 \Bigl( \frac{n}{2} \Bigr) = 0 & if $n$ is even. 
\end{dcases*}
\end{align*}
Thus, we see that the sequences $\{ \tilde{p}_j (n) \}_{n \ge 0}$ $(j=1,2,3)$
are linearly dependent over $\Q(q)$ if and only if the condition \eqref{5.2.con} holds.
Therefore, by Theorem \ref{thm2}, the proof of Example \ref{exm5} is completed.
\end{proof}


\noindent
{\bf Acknowledgments}. 
I would like to express my sincere gratitude to Y. Tachiya for his comments and suggestions. 
I am also very grateful to the referee for valuable comments and for careful reading of my manuscript.


\end{document}